%% file: nuf-matrices.tex
\documentclass[a4paper, reqno]{amsart}
\newif\iffinal
\finaltrue

\pdfoutput=1 

\usepackage[utf8]{inputenc}
\usepackage[T1]{fontenc}
\usepackage{lmodern}
\usepackage[english]{babel}
\usepackage[shortlabels]{enumitem}
\usepackage{amssymb}
\usepackage{colonequals}
\usepackage{mathtools}
\usepackage{bm}
\usepackage[pdfencoding=auto, colorlinks=true, linkcolor=blue, citecolor=blue]{hyperref}
\usepackage{colortbl}
\usepackage{color}
\usepackage{letltxmacro}
\usepackage{todonotes}
\LetLtxMacro\todonotestodo\todo
\renewcommand{\todo}[2][]{\todonotestodo[backgroundcolor=yellow, #1]{TODO: {#2}}}
\newcommand{\attention}[2][]{\todonotestodo[backgroundcolor=red, #1]{TODO: {#2}}}
\iffinal\renewcommand{\todo}[2][]{}\renewcommand{\attention}[2][]{}\fi
\iffinal
\else
\usepackage[mark]{gitinfo2}
\renewcommand{\gitMark}{\jobname\,\textbullet{}\,\gitFirstTagDescribe\,\textbullet{}\,\gitAuthorName,\,\gitAuthorIsoDate}
\fi
\DeclareMathOperator{\Int}{Int}

\DeclareMathOperator{\fdk}{fd-ker}

\DeclareMathOperator{\fixdiv}{d}
\DeclareMathOperator{\rank}{rank}
\def\card#1{\left|#1\right|}
\DeclarePairedDelimiter{\norm}{\lVert}{\rVert}

\newcommand{\Qspan}{\operatorname{span}_{\Q}}
\newcommand{\N}{\mathbb{N}}
\newcommand{\Q}{\mathbb{Q}}

\newcommand{\Z}{\mathbb{Z}}
\newcommand{\val}{\mathsf{v}}

\newcommand{\argdot}{\,\cdot\,}

\newcommand{\Qx}{\mathbb{Q}[x]}

\newcommand{\abs}[1]{\lvert#1\rvert}

\newcommand{\tuple}[1]{\boldsymbol{#1}}
\newcommand{\prodtuple}[2]{#1^{\tuple{#2}}}
\newcommand{\prodnotuple}[2]{#1^{#2}}
\newcommand{\ip}[2]{\langle #1, #2\rangle}
\newcommand{\polyset}{\mathcal{P}}
\newcommand{\witnesses}{\mathcal{W}}
\newcommand{\longmid}{\,\,\middle\vert\,\,}

\newcommand{\Zinfnorm}[1]{\big|\!\big\lceil #1 \big\rceil\!\big|_\infty}
\newcommand{\Zinfnormb}[1]{\bigg|\!\bigg\lceil #1 \bigg\rceil\!\bigg|_\infty}

\makeatletter
\newtheorem*{rep@theorem}{\rep@title}
\newcommand{\newreptheorem}[2]{%
\newenvironment{rep#1}[1]{%
 \def\rep@title{#2 \ref{##1}}%
 \begin{rep@theorem}}%
 {\end{rep@theorem}}}
\makeatother

\makeatletter
\@namedef{subjclassname@2020}{\textup{2020} Mathematics Subject Classification}
\makeatother

\newreptheorem{theorem}{Theorem}

\newtheorem{theorem}{Theorem}
\newtheorem{proposition}{Proposition}[section]

\newtheorem{lemma}[proposition]{Lemma}
\theoremstyle{definition}
\newtheorem{definition}[proposition]{Definition}
\newtheorem{example}[proposition]{Example}

\newtheorem{notation}[proposition]{Notation}
\newtheorem{convention}[proposition]{Convention}
\newtheorem{fact}[proposition]{Fact}
\theoremstyle{remark}
\newtheorem{remark}[proposition]{Remark}
\newtheorem*{rem*}{Remark}
\numberwithin{equation}{section}

\author{Moritz Hiebler}
\address{Department of Mathematics\\University of Klagenfurt\\
  Universitätsstraße 65-67\\9020 Klagenfurt am Wörthersee\\Austria}
\email{\href{mailto:moritz.hiebler@gmx.at}{moritz.hiebler@gmx.at}}

\author{Sarah Nakato}
\address{Department of Mathematics\\
  Kabale University\\
  Plot 364 Block 3 Kikungiri Hill\\Kabale\\Uganda}
\email{\href{mailto:snakato@tugraz.at}{snakato@kab.ac.ug}}
\thanks{S.~Nakato is supported by the Austrian Science Fund (FWF):
P~30934}

\author{Roswitha Rissner}
\address{Department of Mathematics\\University of Klagenfurt\\
  Universitätsstraße 65-67\\9020 Klagenfurt am Wörthersee\\Austria}
\email{\href{mailto:roswitha.rissner@aau.at}{roswitha.rissner@aau.at}}
\thanks{R.~Rissner is supported by the Austrian Science Fund (FWF):
 DOC~78}

\title[%
  \textit{C\MakeLowercase{haracterizing absolutely irreducible int.-valued polynomials
  over discrete valuation domains}}
]{%
  Characterizing absolutely irreducible integer-valued polynomials
  over discrete valuation domains}

\begin{document}

\begin{abstract}
  \input{abstract.tex}
\end{abstract}

\maketitle

\section{Introduction}
\input{intro.tex}

\section{Results}
\label{sec:results}
\input{results.tex}

\section{Preliminaries}
\label{sec:prelim}
\input{preliminaries.tex}

\section{The fixed divisor kernel of a polynomial}
\label{sec:fdkernel}
\input{fdkernel.tex}

\section{Bounds for the minimal power of a non-absolutely irreducible
\label{sec:bounds}
	polynomial to factor non-uniquely}
\input{bounds.tex}

\section{Tightness of the bounds}
\label{sec:tightness}
\input{bound-tightness.tex}

\bibliographystyle{amsplainurl}
\bibliography{biblio}

\end{document}

%% file: abstract.tex
Rings of integer-valued polynomials are known to be atomic,
non-factorial rings furnishing examples for both irreducible elements
for which all powers factor uniquely (\emph{absolutely irreducibles})
and irreducible elements where some power has a factorization
different from the trivial one.

In this paper, we study irreducible polynomials $F \in \Int(R)$ where
$R$ is a discrete valuation domain with finite residue field and
show that it is possible to explicitly determine a number $S\in \N$
that reduces the absolute irreducibility of $F$ to the unique
factorization of $F^S$.

To this end, we establish a connection between the factors of powers of
$F$ and the kernel of a certain linear map that we associate to $F$.
This connection yields a characterization of absolute
irreducibility in terms of this so-called \emph{fixed divisor
  kernel}. Given a non-trivial element $\tuple{v}$ of this kernel, we
explicitly construct non-trivial factorizations of $F^k$, provided that
$k\ge L$, where $L$ depends on $F$ as well as the choice of
$\tuple{v}$. We further show that this bound cannot be improved in
general. Additionally, we provide other (larger) lower bounds for $k$, one of
which only depends on the valuation of the denominator of $F$ and the
size of the residue class field of $R$.


%% file: intro.tex
In atomic domains which allow non-unique factorizations, powers of
irreducible elements $c$ may or may not have a factorization other
than the trivial one, that is,~$c\cdots c$. Those irreducible elements all of
whose powers (essentially) factor uniquely are the so-called
\emph{absolutely irreducible} elements, a notion that ``bridges'' the
gap between prime and irreducible elements.  A thorough understanding
of the factorization behavior of such a ring necessarily requires 
comprehension of its irreducible elements and the factorization behavior of
their powers.

Rings of integer-valued polynomials,
\begin{equation*}
  \Int(D) = \{F \in K[x] \mid F(D) \subseteq D\}
\end{equation*}
where $D$ is a domain with quotient field $K$, are known to provide
examples for both absolutely irreducible elements and irreducible
elements that are not absolutely irreducible, see Angermüller's
recent publication~\cite{Angermuller:2022:strongatoms} and, for
example, Nakato's article~\cite{Nakato:2020:Non-Abs} for explicit
constructions in $\Int(\Z)$.  Note that in the literature, absolutely
irreducible elements have also been called completely
irreducible~\cite{Kaczorowski:1981:compl-irred} and strong
atoms~\cite{Baginski-Kravitz:2010:HFKR,
  Chapman-Krause:2012:Atomic-decay}.

In this paper, we study factorizations of powers of (non-constant)
irreducible elements $F\in \Int(R)$ where $(R,pR)$ is a discrete
valuation domain with finite residue field. In order to prove that $F$ is
not absolutely irreducible, it suffices to allege a non-trivial
factorization of some power $F^k$. Proving absolute irreducibilty, on
the other hand, requires to get a handle on the (unique)
factorizations of all exponents $k$ for which the powers $F^k$
potentially factor non-uniquely.  We establish lower bounds
$S$ such that whenever $F$ is not absolutely irreducible, then $F^k$
factors non-uniquely for $k\ge S$. In other words, to determine
whether $F$ is absolutely irreducible or not, it suffices to check the
factorization behavior of $F^S$, that is, we only need to treat one
power of $F$ instead of infinitely many. Our approach yields a full
characterization of all absolutely irreducible polynomials in
$\Int(R)$ in terms of the kernel of a certain linear map, which we
call the \emph{fixed divisor kernel}.

The focus of factorization theoretic studies has layed on Krull monoids
to a large extent so far. For integral domains, Krull domains are
exactly those domains whose multiplicative monoid is Krull.

For a discrete valuation domain $R$ with finite residue field,
the ring $\Int(R)$ is not Krull but Prüfer,
cf.~\cite{Cahen-Chabert-Frisch:2000:interpolation,
  Loper:1998:intdpruefer}.  Reinhart~\cite{Reinhart:2014:monadic},
however, showed that they are monadically Krull, that is, he proved
that for each $F \in \Int(D)$ (for factorial domains~$D$), the monadic
submonoid generated by $F$,
\begin{equation*}
  [\![F]\!] = \{G \in \Int(D) \mid G \text{ divides } F^n
  \text{ for some } n\in \N \},
\end{equation*}
is Krull and Frisch~\cite{Frisch:2014:monadic} extended this result to
Krull domains. Note that an irreducible polynomial $F \in \Int(D)$
is absolutely irreducible if and only if $[\![F]\!]$ is factorial.

Factorization-theoretic properties of rings of integer-valued
polynomials have been studied over the last decades. The papers of
Anderson, Cahen, Chapman and
Smith~\cite{AndersonS-Cahen-Chapman-Smith:1995:fac-iv}, Cahen and
Chabert~\cite{Cahen-Chabert:1995:Elasticity-for-IVP} and Chapman and
McClain~\cite{Chapman-McClain:2005:irred-iv-poly} can be considered as
the starting point of this line of research and have encouraged
further research activity in the area,
e.g.,~\cite{Antoniou-Nakato-Rissner:2018:table-crit,
  Fadinger-Frisch-Windisch:2022:setsIntV, Frisch:2013:prescribed-SL,
  Frisch-Nakato-Rissner:2019:Sets-of-lengths,
  Peruginelli:2015:square-free-denom}.

Recently, there has been perceptible progress in the study of
absolutely irreducible elements in rings of integer-valued
polynomials. In addition to the already mentioned references, Rissner
and Windisch~\cite{Rissner-Windisch:2021:binom} have confirmed the
decade-long open conjecture that the binomial polynomials
$\binom{x}{n}$ are absolutely irreducible for all $n\in \N$. The
special case where $n$ is a prime number has been verified before by
McClain~\cite{McClain:2004:honorsthesis} and also follows from Frisch
and Nakato's graph-theoretic
criterion~\cite{Frisch-Nakato:2019:Graphtheoretic}. The latter
provides a characterization of absolutely irreducible integer-valued
polynomials over principal ideal domains whose denominators are
square-free. One of the consequences of this criterion is that for
such polynomials $F$ one can determine absolute irreducibility by
checking the unique factorizations of $F^3$. For integer-valued
polynomials whose denominators contain square factors, this criterion
is known to be a sufficient, but not a necessary condition for
absolute irreducibility. Adding to an overall understanding of
absolutely irreducible integer-valued polynomials is the
characterization of the class of completely split absolutely
irreducible elements of $\Int(R)$ where $R$ is a discrete valuation
domain $R$ with finite residue field by Frisch, Nakato, and
Rissner~\cite{Frisch-Nakato-Rissner:2022:split}.

The last two references form the motivational starting point for the
paper at hand. We fully characterize the absolutely irreducible
elements in $\Int(R)$ where $R$ is a discrete valuation domain with
finite residue field and we determine different (explicit) exponents
$k$ such that $F^k$ factors uniquely if and only if $F$ is absolutely
irreducible. This is not only interesting from a theoretical point of
view but also provide means to approach the subject from a
computational perspective.

Note that while the bound resulting from the graph-theoretic
criterion~\cite{Frisch-Nakato:2019:Graphtheoretic} cannot be improved
for integer-valued polynomials over principal ideal domains, it is not
tight in the case when the underlying ring is a discrete valuation
domain $R$ with finite residue field. Indeed, in the case of a single
prime element $p$ in the denominator, it even suffices to check the
factorizations of $F^2$ to determine absolute irreducibility. That is,
if $F = \frac{f}{p} \in \Int(R)$ is irreducible but not absolutely
irreducible, then, by \cite[Theorem~1
and~3]{Frisch-Nakato:2019:Graphtheoretic}, there exists a non-constant
irreducible divisor $h$ of $f$ in $R[x]$ such that for all roots $a$
of $h$ modulo $p$, we have $\val(f(a)) >1$.  Therefore, if $f=gh$ with
$g\in R[x]$, we obtain that
\begin{equation*}
  F^2 = \frac{g^2h}{p^2} \cdot h
\end{equation*}
is a factorization (not necessarily into irreducibles) of $F^2$
different from $F \cdot F$.


%% file: results.tex
As \(\Int(R)\) is trivial for discrete valuation domains \(R\) with
infinite residue field, we restrict our attention to those with finite
residue field, cf.~\cite[Corollary~I.3.7]{Cahen-Chabert:1997:book}.

Given a (non-constant) irreducible $F \in \Int(R)$, we determine an
exponent \(S \in \N\) such that $F$ is absolutely irreducible if and
only if $F^S$ factors uniquely. We show that integer-valued factors of
powers of $F$ are encoded as non-trivial elements of what we call the
\emph{fixed divisor kernel}. This allows us not only to determine
lower bounds for $S$ (non-unique factorizations trivially transfer to
higher powers), but also yields a neat characterization of absolute
irreducibility in terms of this special kernel.

As usual in factorization theory, we are only interested in
essentially different factorizations, which is why we will not
distinguish between associated elements
(cf.~Remark~\ref{remark:polyset-egal}). Hence, it suffices to consider
polynomials $F\in \Int(R)$ which are of the form $F = \frac{f}{p^n}$
with \emph{fixed divisor} $\fixdiv(f) = p^n$
(Definition~\ref{definition:fixed-divisor}) where
$f = \prod_{g\in\polyset}g^{m_g} \in R[x]$ with \emph{irreducible
  divisor set} $\polyset$
(Definition~\ref{definition:irreducible-divisor-set}) and the vector
of the corresponding multiplicities
$\tuple{m} = (m_g)_{g\in \polyset}\in \N^{\polyset}$.

An integer-valued factor of $F^j$ always has to be of the form
$H=\frac{\prod_{g\in \polyset}g^{k_g}}{p^r}$ where $r \le jn$
(Fact~\ref{fact:iv-fac}). Roughly speaking, whether or not $F^j$ has a
non-trivial factorization is asking whether it is possible to
``suitably re-distribute'' the $jm_g$ respective copies of the
polynomials $g$. This, in turn, heavily depends on the values of
$(\val(g(a)))_{g\in\polyset}$ for all $a\in R$.

However, the relevant information about all these valuation vectors
can be encoded in the kernel of a certain linear map, the \emph{fixed
  divisor kernel \(\fdk(f)\)} of $f$
(Definition~\ref{definition:fd-kernel}).

Indeed, in Section~\ref{sec:fdkernel}, we show how to use a
non-trivial element $\tuple{v}\in \fdk(f)$ to explicitly engineer a
non-trivial factorization of $F^k$ for $k \ge L$ where $L$ is a bound
depending on $n$, $\tuple{v}$, and the vector of multiplicities
$\tuple{m}$, leading to our first main result. Regarding notation,
$\tuple{v}^+$ and $\tuple{v}^-$ are the positive and negative part of
$\tuple{v}$, respectively, and
$\Zinfnorm{\tuple{x}} = \lceil \max_g|x_g| \rceil$.

\begin{theorem}\label{theorem:nuf-bound}
  \input{theorem-1.tex}
\end{theorem}

We also show that the reverse implication holds, that is, whenever the
fixed divisor kernel of $f$ is trivial, then $F$ is absolutely
irreducible (where we need to impose a condition on $f$ which is
trivially satisfied whenever $F$ is irreducible). This yields our next
main result, a characterization of absolutely irreducible polynomials
in $\Int(R)$ in terms of the fixed divisor kernel.

\begin{theorem}\label{theorem:abs-irred-and-kernel}
  \input{theorem-2.tex}
\end{theorem}

In the remaining paper, we have a closer look at the bound given in
Theorem~\ref{theorem:nuf-bound}.  In Section~\ref{sec:bounds}, we are
able to give an upper bound for
\begin{equation*}
  \min\!\left\{\norm{\tuple{v}}_\infty \longmid \tuple 0 \neq \tuple{v}\in\fdk(f)\cap \Z^{\polyset}\right\}
\end{equation*}
using a tailored version of Siegel's lemma
(Lemma~\ref{lemma:siegel-version}). For this purpose, we introduce the
notion of a \emph{reduced fdp matrix} $A \in \Q^{W\times \polyset}$
which is a full row-rank matrix satisfying $\ker(A) = \fdk(f)$, where
$W$ is a certain finite set of so-called \emph{fixed divisor
  witnesses} of $f$ (Definitions~\ref{definition:fd-kernel}
and~\ref{definition:reduced-fdp}).  This further allows us to
determine other exponents $S$, one of which depends only on $n$ and the
size of the finite residue field of $R$ such that an irreducible
polynomial $F = \frac{f}{p^n}\in \Int(R)$ is absolutely irreducible if
and only if $F^S$ factors uniquely.

\begin{theorem}\label{theorem:upperbounds}
  \input{theorem-3.tex}
\end{theorem}

Finally, in Section~\ref{sec:tightness}, we discuss the tightness of
the given bounds. We show that the bound in
Theorem~\ref{theorem:nuf-bound} cannot be improved in general. Indeed,
for any $n\ge 2$, there exists a discrete valuation domain $(R,pR)$
with finite residue field and an irreducible, integer-valued
polynomial $F = \frac{f}{p^n}\in \Int(R)$ such that the minimal $S$
for which $F^S$ factors non-uniquely is exactly the bound given in
Theorem~\ref{theorem:nuf-bound}, minimized over all feasible
$\tuple{v}$. We are not only able to determine $S$ explicitly, but
also show that it can be made larger than any predefined constant. We
point out that, given the numbers in Theorem~\ref{theorem:upperbounds},
it follows that the size of the residue field has to grow with the
size of $S$.

\begin{theorem}\label{theorem:realization}
  \input{theorem-4.tex}
\end{theorem}


%% file: theorem-1.tex
Let $(R,pR)$ be a discrete valuation domain with valuation $\val$ and
finite residue field.  Further, let
$f = \prod_{g\in \polyset}g^{m_g}\in R[x]$ be a primitive,
non-constant polynomial with irreducible divisor set $\polyset$ and
\(\tuple{m} = (m_g)_{g\in \polyset}\in\N^\polyset\) the vector of the
corresponding multiplicities and assume that
$\val(\fixdiv(f)) = n \in \N$.

If $\fdk(f) \neq \boldsymbol 0$, then $F = \frac{f}{p^n}$ is not absolutely
irreducible. In fact, if $F$ is irreducible and
$\tuple 0 \neq \tuple v \in \fdk(f) \cap \Z^{\polyset}$, then $F^{j}$ factors
non-uniquely for all $j\in \N$ with
$j \ge (n+1) \left(\Zinfnorm{\frac{\tuple{v}^+}{\tuple{m}}}
  +
 \Zinfnorm{\frac{\tuple{v}^-}{\tuple{m}}}\right)$.


%% file: theorem-2.tex
Let $(R,pR)$ be a discrete valuation domain with valuation $\val$ and
finite residue field. Further, let $f \in R[x]$ be a primitive,
non-constant polynomial with $\val(\fixdiv(f)) = n \in \N$ and assume
that $f$ is not a proper power of another polynomial in $R[x]$.

Then $\frac{f}{p^n}$ is absolutely irreducible if and only if
$\fdk(f) = \boldsymbol 0$.


%% file: theorem-3.tex
Let $(R,pR)$ be a discrete valuation domain with valuation $\val$ and
let $q = \card{R/pR}$ be the cardinality of the finite residue field
of $R$.  Further, let $f\in R[x]$ be a
non-constant, primitive polynomial with irreducible divisor set
$\polyset$ and \((m_g)_{g\in \polyset}\in\N^\polyset\) the
vector of the corresponding multiplicities, that is,
$f = \prod_{g\in \polyset}g^{m_g}$. Let $\val(\fixdiv(f)) = n \in \N$
and $A \in \Q^{W\times \polyset}$ (with $W \subseteq\witnesses(f)$) be
a reduced fdp matrix of $f$ containing $u$ rows with only one non-zero
entry.

Assume that $F$ is irreducible. Then the following assertions are equivalent:
\begin{enumerate}
\item\label{bound:1} $F^j$ factors uniquely for all $j\in \N$, that
  is, $F$ is absolutely irreducible.
\item\label{bound:2} $F^S$ factors uniquely for
  $S = 2(n+1)n^{q^{\left\lceil \frac{n}{2} \right\rceil}}$.
\item\label{bound:3} $F^S$ factors uniquely for
  $S = 2(n+1) n^{\rank(A)-u}$.
\end{enumerate}


%% file: theorem-4.tex
Let $r$, $n \ge 2$ be integers.

Then there exists a discrete valuation domain $(R,pR)$ with finite
residue field and a polynomial $F = \frac{f}{p^n}\in \Int(R)$ which is
irreducible, but \textbf{not} absolutely irreducible in \(\Int(R)\)
(both $R$ and $F$ depending on $r$) such that the minimal
exponent $S$ for which $F^S$ does not factor uniquely satisfies

\begin{enumerate}
\item\label{thm:bounds}
  $S = (n+1)\left((n-1)^{r-1} + (n-1)^{r-2}\right)$ where $r$ is the
  rank of a (reduced) fdp matrix of $f$ and
\item\label{thm:K-thm1}
  $S= (n+1)\min\!\left\{\norm{\tuple{v}^+}_\infty +
  \norm{\tuple{v}^-}_\infty \longmid \tuple 0\neq \tuple{v} \in \fdk(f)\cap
  \Z^\polyset\right\}$.
\end{enumerate}
In particular, it follows that the lower bound given in
Theorem~\ref{theorem:nuf-bound} cannot be improved in general.


%% file: preliminaries.tex
\subsection{Factorizations}
We define the factorization terms that we need in this paper, and
refer to the textbook of Geroldinger and
Halter-Koch~\cite{Geroldinger-HalterKoch:2006:nuf} for a
systematic introduction to non-unique factorizations.

\begin{definition}
  Let $R$ be a commutative ring with identity.
  \begin{enumerate}
  \item We say that $r\in R$ is \emph{irreducible} in $R$ if it is a
    non-zero non-unit and it cannot be written as the product of two
    non-units of $R$.
  \item A \emph{factorization} of $r \in R$ is an expression of $r$ as
    a product of irreducibles, that is,
    \begin{equation*}
      \label{eq:fac} r = a_{1}\cdots a_{n}
    \end{equation*}
    where $n\ge 1$ and $a_i$ is irreducible in $R$ for $1\le i \le n$.
  \item We say that $r$, $s\in R$ are \emph{associated} in $R$ if
    there exists a unit $u \in R$ such that $r = us$. We denote this
    by $r \sim s$.
  \item Two factorizations of the same element,
    \begin{equation}\label{eq:2-fac-same-diff}
      r = a_{1}\cdots a_{n} = b_{1} \cdots b_{m},
    \end{equation}
    are called \emph{essentially the same} if $n = m$ and, after a
    suitable re-indexing, $a_{j}\sim b_{j}$ for $1 \leq j \leq m$.
    Otherwise, the factorizations in \eqref{eq:2-fac-same-diff} are
    called \emph{essentially different}.
  \end{enumerate}
\end{definition}

\begin{definition}\label{defabsirred}
  Let $R$ be a commutative ring with identity.  An irreducible element
  $r\in R$ is called \emph{absolutely irreducible} if for all natural
  numbers $n$, every factorization of $r^n$ is essentially the same as
  $r^n = r \cdots r$.
\end{definition}
\begin{remark}\label{remark:equiv-abs-irred}
  A straight-forward verification shows that an irreducible element
  $r$ is absolutely irreducible if and only if for every $n\in \N$ and
  every factorization $r^n = c\cdot d$ into the product of (not
  necessarily irreducible) elements $c$ and $d\in R$, it follows that
  $c \sim r^k$ and $d \sim r^{\ell}$ for some $k$, $\ell\in \N_0$.
\end{remark}

\subsection{(Integer-valued) polynomials}
We shortly summarize definitions, notation and facts surrounding
(integer-valued) polynomials over a discrete valuation domain~$(R,pR)$
which we need throughout this paper. For a deeper study of the theory
of integer-valued polynomials, we refer to the textbook of Cahen and
Chabert~\cite{Cahen-Chabert:1997:book} and their recent
survey~\cite{Cahen-Chabert:2016:survey}.

\begin{definition}\label{definition:irreducible-divisor-set}
  Let $f \in R[x]$ be a polynomial.
  \begin{enumerate}
  \item We call $f$ \emph{primitive} if the coefficients of $f$
    generate $R$ as ideal.
  \item We call a representative set $\polyset$ of the associate
    classes of the irreducible divisors of $f$ an \emph{irreducible
      divisor set of $f$}.
  \end{enumerate}
\end{definition}

\begin{definition}\label{definition:fixed-divisor}
  Let $K$ be the quotient field of $R$. The \emph{ring of
    integer-valued polynomials} on $R$ is
  \begin{equation*}
    \Int(R) = \left\{F \in K[x] \mid F(R) \subseteq R\right\}.
  \end{equation*}
  For $F\in \Int(R)$, the \emph{fixed divisor} of $F$
  is defined as
  \begin{equation*}
    \fixdiv(F) = \gcd\!\left(F(a) \mid a\in R\right).
  \end{equation*}
\end{definition}

\begin{remark}\label{remimageprimitive}
  Every polynomial $F\in \Int(R)$ is of the form $F=\frac{f}{p^n}$ for
  some $f\in R[x]$ and $n\in \N$.
  \begin{enumerate}
  \item  $F \in \Int(R)$ if and only if $p^n \mid \fixdiv(f)$.
  \item If $F \in \Int(R)$ is irreducible in $\Int(R)$, then
    $\fixdiv(f) = p^n$.
  \end{enumerate}
\end{remark}

\begin{convention}\label{convention:R}
  Throughout this paper, unless explicitly stated otherwise, let
  $(R,pR)$ be a discrete valuation domain with valuation $\val$ and
  let $q = \card{R/pR}$ be the cardinality of the finite residue field
  of $R$.

  Further, let $f\in R[x]$ be a non-constant, primitive polynomial
  with $\fixdiv(f) = p^n$ and irreducible divisor set $\polyset$ and
  let $F = \frac{f}{p^n}$.
\end{convention}

As we encounter products of a set of irreducible elements quite
frequently, we adopt a new abbreviated notation to emphasize
the focus on the exponents (inspired by the
standard notation $X^{\boldsymbol{a}}$ in multivariate polynomial rings).
\begin{notation}\label{notation:set-powers}
  Let \(\polyset\subseteq R[x]\) be a non-empty, finite set of
  polynomials. For any
  \(\tuple{m}=(m_g)_{g\in \polyset}\in\N_0^\polyset\), we write
  \begin{equation*}
    \prodtuple{\polyset}{m}= \prod_{g\in \polyset} g^{m_g}.
  \end{equation*}
\end{notation}

\begin{remark}\label{remark:polyset-egal}
  Using the notation of Convention~\ref{convention:R}, let
  $\tuple{m} = (m_g)_{g\in\polyset}\in \N^{\polyset}$ be the vector of
  the multiplicities with which $g\in \polyset$ occur as factors of $f$.

  Then, considering the fact that the units of $\Int(R)$ are exactly
  the units of $R$,
  \begin{equation*}
    F = \frac{f}{p^n} \sim \frac{\prodtuple{\polyset}{m}}{p^n} = \frac{\prod_{g\in \polyset}g^{m_g}}{p^n}
  \end{equation*}
  and the (essentially different) factorizations of $F$ correspond to
  the (essentially different) factorizations of
  $\frac{\prod_{g\in \polyset}g^{m_g}}{p^n}$. Having this in mind, we
  restrict our investigation to polynomials of the latter form for
  simplicity.
\end{remark}

The following well-known fact deals with the factorizations at
issue. It follows from the proof {\cite[Theorem
  2.8]{Chapman-McClain:2005:irred-iv-poly}}, which takes advantage of
the fact that $\Int(R)$ is a subring of the polynomial ring over the
quotient field of $R$, where we encounter unique factorization.
\begin{fact}[{\cite[Theorem 2.8]{Chapman-McClain:2005:irred-iv-poly}}]\label{fact:iv-fac}
  Let $(R,pR)$ be a discrete valuation domain with finite residue
  field and $F \in \Int(R)$ of the form
  \begin{equation*}
    F \sim \frac{\prodtuple{\polyset}{m}}{p^n} = \frac{\prod_{g\in \polyset}g^{m_g}}{p^n}
    \quad\text{ with }\quad
    \fixdiv\!\left(\prodtuple{\polyset}{m}\right) =  p^n
  \end{equation*}
  where $\tuple{m} = (m_g)_{g\in \polyset}\in \N^{\polyset}$ for
  $g\in \polyset$, $n\in\N_0$, and $\polyset\subseteq R[x]$ is a
  non-empty, finite set of irreducible, non-constant polynomials which
  are pairwise non-associated.

  If $F=F_1 \cdots F_r$ is a factorization of $F$ into (not
  necessarily irreducible) non-units in $\Int(R)$, then, for each
  $1\le j \le r$,
  \begin{equation*}
    F_j\sim  \frac{\prodtuple{\polyset}{m_j}}{p^{k_j}} = \frac{\prod_{g \in \polyset_j}g^{m_{j,g}}}{p^{k_j}}
  \end{equation*}
  where $\emptyset\neq \polyset_j\subseteq \polyset$,
  $\tuple{m_j}= (m_{j,g})_{g\in\polyset}\in \N_0^{\polyset}$, and
  $k_j\in \N_0$ such that $\sum_{j=1}^rk_j = k$ and
  $\sum_{j=1}^rm_{j,g} = m_g$ for all $g\in \polyset$.
\end{fact}

We conclude this section with a straight-forward observation on the
properties of Notation~\ref{notation:set-powers}.

\begin{remark}\label{rem:calc-w-prodtuple}
  Let \(\tuple{\ell}=(\ell_g)_{g\in \polyset}\),
  \(\tuple{m}=(m_g)_{g\in \polyset}\in \N_0^\polyset\). Using the
  notation of Convention~\ref{convention:R}, we infer
  \begin{equation*}
    \prodtuple{\polyset}{\ell} = \prodtuple{\polyset}{m} \iff
      \tuple{\ell} = \tuple{m}
  \end{equation*}
  as well as
  \begin{equation*}
    \prodtuple{\polyset}{\ell} \cdot \prodtuple{\polyset}{m} =
    \prodtuple{\polyset}{\ell+m}
    \quad\text{and}\quad
    (\prodtuple{\polyset}{m})^j = \prodnotuple
    \polyset{j\tuple{m}} \quad\text{for }j\in \N_0
  \end{equation*}
  from the definition and the fact that \(R[x]\) is factorial.

  Note also that all divisors of \(\prodtuple{\polyset}{m}\) in
  \(R[x]\) are given by the elements of \(R[x]\) which are associated
  to \(\prodtuple{\polyset}{k}\) for some
  \(\tuple{k}\in \N_0^\polyset\) with \(\tuple{k}\le \tuple{m}\)
  componentwise.
\end{remark}


%% file: fdkernel.tex
\begin{definition}\label{definition:fd-kernel} Using the notation of
  Convention~\ref{convention:R}, we define the \emph{set of fixed
    divisor witnesses of $f$} as
  \begin{equation*}
    \witnesses(f) =  \{a \in R \mid \val(f(a)) = n\},
  \end{equation*}
  where we recall that $n=\val(\fixdiv(f))$.
  Further, we define the \emph{fixed divisor kernel of $f$} to be
  \begin{equation*}
    \fdk(f) = \bigcap_{a\in \witnesses(f)}\ker\!\left( \tuple{m}\in \Q^{\polyset} \mapsto \ip{\tuple{m}}{\val_{\polyset}(a)} \right)
  \end{equation*}
  where $\ip{\argdot}{\argdot}$ is the standard inner product and
  $\val_{\polyset}(a) = (\val(g(a)))_{g\in \polyset}$ for $a\in
  R$. (Note that the choice of $\polyset$ only affects the indexing,
  but not the values of valuations.)
\end{definition}

\begin{remark}\label{rem:prodtuple}
  Unraveling the definition yields
  $\ip{\tuple{m}}{\val_{\polyset}(a)} = \sum_{g\in \polyset}
m_g\val(g(a))$ as well as
  \begin{equation*}
    \fdk(f) = \bigg\{ \tuple{m} = (m_g)_{g\in \polyset}\in \Q^{\polyset} \,\,\big\vert \,\, \forall a\in \witnesses(f)\colon \sum_{g\in \polyset} m_g\val(g(a)) = 0  \bigg\}.
  \end{equation*}
\end{remark}

\begin{remark}\label{remark:equal-rows}
  The following observations are easily verified:
  \begin{enumerate}
  \item Whenever $r$, $s\in R$ with $\val(r-s)\ge \val(r) + 1$, then
    $\val(s) = \val(r)$.
  \item Let $g\in R[x]$,  $a$, $w\in R$, and $n\in \N$. If
    $\val(a-w)\ge n$, then $\val(g(a) - g(w)) \ge n$.
  \item\label{ri:same-val} The combination of the items above implies
    that whenever $\val(a-w) \ge \val(g(a)) + 1$, then
    $\val(g(a)) = \val(g(w))$.
  \end{enumerate}
\end{remark}

By Remark~\ref{remark:equal-rows}\ref{ri:same-val}, the set
$\witnesses(f)$ of witnesses is always infinite, since
$\witnesses(f)\neq\emptyset$ and with every $w\in \witnesses(f)$ the
(infinite) residue class $w + p^{\val(\fixdiv(f))+1}R$ is a subset of
$\witnesses(f)$.

However, as a subspace of $\Q^{\polyset}$, $\fdk(f)$ has finite
dimension and hence $\fdk(f)$ is the kernel of a linear map
$\Q^\polyset \to \Q^{W}$ for a finite set $W \subseteq \witnesses(f)$,
or equivalently, it can be represented as the kernel of a matrix.

\begin{definition}\label{def:fixed-divisor-partition-matrix}
  Let $W$ be a finite set. With the notation of
  Convention~\ref{convention:R}, we say that a matrix
  $A \in \Q^{W \times \polyset}$ is an \emph{fdp~matrix} (short for
  \emph{fixed divisor partition matrix}) of $f$ if $\ker(A) = \fdk(f)$
  (where we consider the kernel of $A$ to be the kernel of the vector
  space homomorphism $\Q^\polyset \to \Q^W$ associated to~$A$).
\end{definition}

Next, we specify how an fdp matrix of $f$ can be
determined.
\begin{lemma}\label{lemma:matrix-rows}
  With the notation of Convention~\ref{convention:R}, let $W$ be a set of
  representatives of the witness set $\witnesses(f)$ modulo
  $p^{\left\lceil\frac{n}{2}\right\rceil}R$.

  Then, for each $a\in \witnesses(f)$, there exists $w\in W$ such that
  $\val_\polyset(a) = \val_\polyset(w)$.

  In particular, the matrix
  $(\val_\polyset(w))_{w\in W} \in \Q^{W\times \polyset}$ is an fdp
  matrix of $f$.
\end{lemma}
\begin{proof}
  Let $w\in W$ be the representative of $a$ modulo
  $p^{\left\lceil\frac{n}{2}\right\rceil}R$. We demonstrate that
  $\val_\polyset(a)= \val_\polyset(w)$ holds.  Remark~\ref{remark:equal-rows}\ref{ri:same-val}
  implies that $\val(g(a)) = \val(g(w))$ provided that
  $\val(g(a)) < \left\lceil\frac{n}{2}\right\rceil$.

  It remains to consider the elements of
  \begin{equation*}
    J = \left\{g \in \polyset \longmid \val(g(a)) \ge \bigg\lceil\frac{n}{2}\bigg\rceil \right\}.
  \end{equation*}
  Let $\tuple{m} = (m_g)_{g\in \polyset} \in \N^\polyset$ be the
  vector of the multiplicities with which the polynomials
  $g\in \polyset$ occur as divisors of $f$ so that
  \(f = \prodtuple{\polyset}{m}\) (see
  Notation~\ref{notation:set-powers}). Then
  \begin{equation}\label{eq:find-row}
    n = \val(f(a)) = \sum_{g\in \polyset}m_g\val(g(a))
  \end{equation}
  holds, which implies that the set $J$ contains at most two elements
  since $m_g\in \N$. We split the remainder of the proof into two
  cases, $|J| = 1$ and $|J|=2$.

  If $J = \{h\}$, then
  \begin{equation*}
    m_h\val(h(a))
    = n -\sum_{g\in \polyset\setminus J}m_g\val(g(a))
    = n -\sum_{g\in \polyset\setminus J}m_g\val(g(w))
    = m_h\val(h(w))
  \end{equation*}
  and hence $\val(h(a)) = \val(h(w))$.

  Finally, we assume that $J = \{h_1, h_2\}$ with $h_1\neq h_2$.  Due
  to Equation~\eqref{eq:find-row}, this is only possible if $n$ is
  even and $\val(h_1(a)) = \val(h_2(a)) = \frac{n}{2}$. Moreover, by
  Remark~\ref{remark:equal-rows}, $\val(h_1(w))$, $\val(h_2(w))$ are
  both at least $\frac{n}{2}$.  As $w$ is a fixed divisor witness, it
  follows that
  \begin{equation*}
    n = \val(f(w)) = \sum_{g\in \polyset}m_g\val(g(w)) \ge m_{h_1} \val(h_1(w)) + m_{h_2}\val(h_2(w)) \ge n
  \end{equation*}
  holds.  We conclude that
  $\val(h_1(w)) = \val(h_2(w)) = \frac{n}{2}$, which proves the claim.
\end{proof}

\begin{example}
  Let $f=(x^2+9)(x-5)^3(x-1)(x-7) \in \Z_{(3)}[x]$ where $\Z_{(3)}$ is
  the localization of $\Z$ at 3. A straight-forward verification shows
  that the fixed divisor of $f$ is 9. In order to find an fdp matrix,
  it suffices to find a set of representatives modulo $3$ of all the
  fixed divisor witnesses of $f$. By evaluating $f$ at $0$ and $4$, we
  can verify that both are in $\witnesses(f)$. Also, the factor
  $(x-5)^3$ of $f$ guarantees that no fixed divisor witness is
  congruent to $2$ modulo $3$. We set $W=\{0,4\}$ and apply
  Lemma~\ref{lemma:matrix-rows} to conclude that
  \begin{equation*}
    \begin{pmatrix}
      2 & 0 & 0 & 0 \\
      0 & 0 & 1 & 1
    \end{pmatrix} \in \Q^{W\times \polyset}
  \end{equation*}
  is a fixed divisor partition matrix of $f$ where
  $\polyset = \{x^2+9, x-5, x-1, x-7\}$ (and for the purpose of
  writing down the matrix, we impose the order on $W$ and $\polyset$
  with which their elements are given here).
\end{example}

In the following, we discuss the connection between $\fdk(f)$ and the question
whether $F = \frac{f}{p^n}$ is absolutely irreducible (where
$\val(\fixdiv(f)) = n$). We start with the connection to
integer-valued divisors of powers of $F$.

\begin{lemma}\label{lemma:factors-and-kernelA}
  Using the notation of Convention~\ref{convention:R}, let
  $f = \prod_{g\in \polyset}g^{m_g} \in R[x]$ where
  \(\tuple{m} = (m_g)_{g\in \polyset}\in\N^\polyset\) is the vector of
  the corresponding multiplicities.

  If $\tuple 0\neq \tuple{k}= (k_g)_{g\in \polyset}\in \N_0^{\polyset}$ and $\ell\in \N_0$ such that
  $\frac{\prod_{g\in \polyset}g^{k_g}}{p^\ell}$ divides $F^j$ in $\Int(R)$ for
  some $j\in \N$, then
  \begin{equation*}
    \tuple{k}\in \frac{\ell}{n}\tuple{m} + \fdk(f).
  \end{equation*}
\end{lemma}
\begin{proof}
  Using the abbreviated notation for products of polynomials,
  cf.~Notation~\ref{notation:set-powers}, we write
  $f = \prodtuple \polyset m$ and set $f_1 = \prodtuple \polyset k$
  and $f_2 = \prodnotuple{\polyset}{j\tuple m - \tuple k}$.  Then
  $\frac{f_2}{p^{jn-\ell}}$ is the cofactor to $\frac{f_1}{p^\ell}$ of $F^j$
  in $\Int(R)$ and $\val(f_1(a)) \ge \ell$ and $\val(f_2(a)) \ge jn-\ell$
  hold for all $a\in R$.

  If $w\in \witnesses(f)$ is a fixed divisor witness of $f$,
  Remark~\ref{rem:prodtuple} yields
  \begin{align*}
    jn
    &= j \val(f(w)) = j\ip{\tuple{m}}{\val_\polyset(w)}
    = \ip{\tuple{k}}{\val_\polyset(w)} + \ip{j\tuple{m} - \tuple{k}}{\val_\polyset(w)} \\
    &= \val(f_1(w)) + \val(f_2(w)) \ge \ell + (jn-\ell) = jn,
  \end{align*}
  implying equality throughout, in particular
  $\ell = \val(f_1(w)) = \ip{\tuple{k}}{\val_\polyset(w)}$.
  Then $\tuple k$ is a solution to the
  linear equation $\ip{\tuple{x}}{\val_\polyset(w)} = \ell$.  Since
  $\frac{\ell}{n}\tuple{m}$ is another solution to it (cf.\
  Remark~\ref{rem:prodtuple}), it follows that
  \begin{equation*}
    \tuple k - \frac{\ell}{n} \tuple m \in
    \ker\!\left( \tuple{x}\in \Q^{\polyset} \mapsto \ip{\tuple{x}}{\val_{\polyset}(w)} \right)
  \end{equation*}
  for all $w\in \witnesses(f)$ and thus
  \(\tuple k - \frac{\ell}{n} \tuple m \in \fdk(f)\).
\end{proof}


Next, we show how we can use non-zero elements in $\fdk(f)$ to
construct specific polynomials in $\Int(R)$, which will turn out to be
non-trivial divisors of $\frac{f}{p^n}$.

\begin{definition}\label{definition:v-plus}
  Let \(\tuple{u}\), \(\tuple{v} \in \Q^\polyset\). The positive and
  negative part of \(\tuple{u}\) are defined by
  \begin{equation*}
    \tuple{u}^+ = \max(\tuple{u}, \tuple{0}) \quad\text{and}\quad \tuple{u}^- = - \min(\tuple{u}, \tuple{0}),
  \end{equation*}
  respectively, where $\max$ and $\min$ are to be understood
  componentwise and \(\tuple{0} = (0)_{g\in \polyset}\) denotes the
  zero vector. We further write
  \(\|\tuple{u}\|_\infty= \max_{g \in \polyset} \abs{u_g}\) for the
  usual infinity norm of \(\tuple{u} = (u_g)_{g\in \polyset}\) and
  \begin{equation*}
    \Zinfnorm{\tuple{u}} = \big\lceil \|\tuple{u}\|_\infty \big\rceil.
  \end{equation*}

  Finally, for $\tuple{u} = (u_g)_{g\in \polyset}$ and
  $\tuple{v} = (v_g)_{g\in \polyset}$ in $\Q^{\polyset}$, we define
  \begin{equation*}
    \tuple{u} \cdot \tuple{v} = (u_g \cdot v_g)_{g\in \polyset} \qquad\text{and}\qquad \frac{\tuple{u}}{\tuple{v}} =
    \left(\frac{u_g}{v_g}\right)_{g\in \polyset},
  \end{equation*}
  where the latter is only defined whenever all $v_g$ are
  non-zero. The unit element with respect to this multiplication
  is \(\tuple{1} = (1)_{g\in \polyset}\).
\end{definition}

\begin{proposition}\label{proposition:in-IntR}
  Using the notation of Convention~\ref{convention:R}, let
  $f = \prod_{g\in \polyset}g^{m_g} \in R[x]$ where
  \(\tuple{m} = (m_g)_{g\in \polyset}\in\N^\polyset\) is the vector of
  the corresponding multiplicities.

  If
  $\tuple 0\neq \tuple{v} = (v_g)_{g\in \polyset} \in \fdk(f) \cap
  \Z^{\polyset}$ and \(k \in \N\) with
  $k\ge (n+1) \Zinfnorm{\frac{\tuple{v}^+}{\tuple{m}}}$, then
  \begin{equation*}
    H = \frac{\prod_{g\in\polyset}g^{k m_g - v_g}}{p^{kn}} \in \Int(R).
  \end{equation*}
  Moreover, $H$ is not a power of $\frac{f}{p^n}$.
\end{proposition}

\begin{remark}\label{remark:exponents-in-detail}
  Set
  $s^+ =
  \Zinfnorm{\frac{\tuple{v}^+}{\tuple{m}}}$ and
  note that all components of \(\tuple{m}\) are positive.
  By definition of $\tuple{v}^+$, the identity
  \begin{equation*}
    s^+
    = \max\!\left\{\left\lceil\frac{v_g}{m_g}\right\rceil \longmid g\in \polyset \text{ with } v_g \ge 0\right\}
  \end{equation*}
  holds and we have $s^+>0$ since
  $\tuple 0\neq \val_{\polyset}(w)\in \N_0^{\polyset}$ for all fixed divisor
  witnesses $w$ of~$f$.

  The definition also immediately implies \(s^+\tuple{m} \ge \tuple{v}^+\).

  Observe further that \(k\tuple{m} - \tuple{v}\) only consists of
  positive integers because the positivity of \(k\) entails either
  \(k > 0 \ge \frac{v_g}{m_g}\), if \(v_g \le 0\), or
  \(k\ge (n+1) s^+ \ge (n+1)\frac{v_g}{m_g} > \frac{v_g}{m_g}\), if
  \(v_g > 0\).
\end{remark}

\begin{proof}
  Again, we switch to the abbreviation for writing polynomial products
  (Notation~\ref{notation:set-powers}) and set
  $\tilde f = \prodnotuple{\polyset}{k \tuple{m} - \tuple{v}}$ to be the
  numerator of $H$.

  We need to show
  \begin{equation}\label{eq:ntfactor}
    \val(\tilde f(a)) \ge kn\qquad \text{for all \(a \in R\).}
  \end{equation}

  Assume first that $\val(h(a)) = \infty$ for some $h\in \polyset$,
  that is, at least one entry of \(\val_\polyset(a)\) is equal to
  \(\infty\). Then the strict positivity of
  \(k\tuple{m} - \tuple{v}\), as observed in
  Remark~\ref{remark:exponents-in-detail}, assures that the component
  of \(k\tuple{m} - \tuple{v}\) corresponding to $h$ is positive and hence
  \begin{equation*}
    \val(\tilde f(a)) = \ip{k\tuple{m} -
      \tuple{v}}{\val_\polyset(a)} = \infty \ge kn.
  \end{equation*}

  For the remainder of the proof we may thus assume that
  $\val(g(a)) < \infty$ for all $g\in
  \polyset$. Then we obtain
  \begin{equation}
  	\begin{aligned}\label{eq:val:tildef}
  		\val(\tilde f(a))
  		= \ip{k \tuple{m} - \tuple{v}}{\val_\polyset(a)}
  		&= k \ip{\tuple{m}}{\val_\polyset(a)} - \ip{\tuple{v}}{\val_\polyset(a)} \\
  		&= k \cdot \val(f(a)) - \ip{\tuple v}{\val_\polyset(a)}
  	\end{aligned}
  \end{equation}
  from Remark~\ref{rem:prodtuple}, as the difference is now well-defined.

  If \(a\) is a fixed divisor witness, then
  \(\ip{\tuple{v}}{\val_\polyset(a)} = 0\), so
  \(\val(\tilde f(a)) = k\val(f(a)) = kn\) by~\eqref{eq:val:tildef}.

  It remains to show~\eqref{eq:ntfactor} for all $a \in R$ with
  $\val(f(a)) \ge n+1$. We claim that
  \begin{equation}\label{eq:needthis}
   k \val(f(a)) \ge kn + \ip{\tuple{v}^+}{\val_\polyset(a)}
  \end{equation}
  holds in that case.

  Assume the validity of~\eqref{eq:needthis} for the moment. Then
  we obtain
  \begin{equation*}
    \val(\tilde f(a)) = k \val(f(a)) - \ip{\tuple{v}}{\val_\polyset(a)}
    \ge kn + \ip{\tuple{v}^+}{\val_\polyset(a)} - \ip{\tuple{v}}{\val_\polyset(a)} \ge kn
  \end{equation*}
  from Equation~\eqref{eq:val:tildef} and \(\tuple{v}^+ \ge \tuple{v}\),
  which finishes the proof of~\eqref{eq:ntfactor}.

  We only need to prove Inequality~\eqref{eq:needthis}. Set
  $s^+ = \Zinfnorm{\frac{\tuple{v}^+}{\tuple{m}}}$,
  as in Remark~\ref{remark:exponents-in-detail}, and
  \(\val(f(a)) = n+j\) for some \(j \in \N\).  Since
  \(k \ge (n+1) s^+\) and \(s^+\tuple{m} \ge \tuple{v}^+\), we conclude
  that
  \begin{align*}
    k \big(\val(f(a))-n\big) = kj
    &\ge (n+1)s^+j = ns^+j + s^+j \\
    &\ge ns^+ +s^+j = s^+(n+j) = s^+\val(f(a))\\
    &= s^+\ip{\tuple{m}}{\val_\polyset(a)} = \ip{s^+\tuple{m}}{\val_\polyset(a)} \ge \ip{\tuple{v}^+}{\val_\polyset(a)},
  \end{align*}
  which is equivalent to~\eqref{eq:needthis}.

  Finally, we show that $H$ is not a power of $F$. Suppose the
  contrary. Then
  \begin{equation*}
    \frac{\prodnotuple{\polyset}{k\tuple{m} -\tuple{v}}}{p^{kn}} = \frac{\prodnotuple{\polyset}{t\tuple{m}}}{p^{tn}}
    \iff p^{tn} \prodnotuple{\polyset}{k\tuple{m} -\tuple{v}} = p^{kn} \prodnotuple{\polyset}{t\tuple{m}}
    \end{equation*}
    for some \(t \in \N_0\), which entails \(t = k\) and
    \(\prodnotuple{\polyset}{k\tuple{m} - \tuple{v}} =
    \prodnotuple{\polyset}{t\tuple{m}} =
    \prodnotuple{\polyset}{k\tuple{m}}\) since \(p\) is also a prime
    element in \(R[x]\). But then \(\tuple{v} = \tuple 0\) since $R[x]$ is
    factorial (cf.~Remark~\ref{rem:calc-w-prodtuple}), contradicting
    the assumption on \(\tuple{v}\).
\end{proof}

We are now ready to prove the main results of this section.

\begin{reptheorem}{theorem:nuf-bound}
  \input{theorem-1.tex}
\end{reptheorem}
\begin{rem*}
  Observe that \((-\tuple{v})^+ = \tuple{v}^-\), so that the second
  summand of the lower bound corresponds to the lower bound from
  Proposition~\ref{proposition:in-IntR} for \(-\tuple{v}\).
\end{rem*}
\begin{proof}
  We can assume that $F$ is irreducible since otherwise it is also not
  absolutely irreducible. Let
  $\tuple 0\neq \tuple{v}\in \fdk(f) \cap \Z^{\polyset}$.  It suffices to
  show that
  $F^j$ with $j={(n+1) \left(\Zinfnorm{\frac{\tuple{v}^+}{\tuple{m}}} +
    \Zinfnorm{\frac{\tuple{v}^-}{\tuple{m}}}\right)}$ factors
  non-uniquely.  Let $k$, $\ell\in \N$ with
  \begin{equation*}
    k\ge (n+1)\Zinfnormb{\frac{\tuple{v}^+}{\tuple{m}}} \quad\text{ and }\quad
    \ell \ge (n+1)\Zinfnormb{\frac{\tuple{v}^-}{\tuple{m}}}.
  \end{equation*}
  Using the abbreviated notation for polynomial products
  (Notation~\ref{notation:set-powers}), it follows from
  Proposition~\ref{proposition:in-IntR} that
  \begin{equation*}
    \frac{\prodnotuple{\polyset}{k\tuple{m} - \tuple{v}}}{p^{kn}} \in \Int(R)
  \end{equation*}
  and that this polynomial is not a power of $F$.

  Similarly, we can apply Proposition~\ref{proposition:in-IntR} with
  $-\tuple{v}$ (which is also a non-zero element of
  $\fdk(f) \cap \Z^{\polyset}$) and conclude that
  \begin{equation*}
    \frac{\prodnotuple \polyset{\ell\tuple{m} - (-\tuple{v})}}{p^{\ell n}}
    = \frac{\prodnotuple \polyset{\ell\tuple{m} +\tuple{v}}}{p^{\ell n}}\in \Int(R),
  \end{equation*}
  again not a power of $F$.  Therefore,
  \begin{equation*}
    F^{k +\ell} = \frac{\prodnotuple{\polyset}{(k+\ell)\tuple{m}}}{p^{(k+\ell)n}}
    = \frac{\prodnotuple{\polyset}{k\tuple{m} - \tuple{v}}}{p^{kn}} \cdot
    \frac{\prodnotuple{\polyset}{\ell\tuple{m} + \tuple{v}}}{p^{\ell n}}
  \end{equation*}
  is a factorization of $F^{k +\ell}$ (not necessarily into
  irreducibles) which yields a factorization of $F^{k+\ell}$ into
  irreducibles different from the trivial one $F \cdot F \cdots F$
  (see Remark~\ref{remark:equiv-abs-irred}). It follows that $F^j$
  factors non-uniquely for every exponent $j$ with
  $j \ge (n+1) \left(\Zinfnorm{\frac{\tuple{v}^+}{\tuple{m}}}
  + \Zinfnorm{\frac{\tuple{v}^-}{\tuple{m}}}\right)$.
\end{proof}

\begin{reptheorem}{theorem:abs-irred-and-kernel}
  \input{theorem-2.tex}
\end{reptheorem}

\begin{proof}
  Assume first \(\fdk(f) = \boldsymbol 0\), let \(\polyset\) be an
  irreducible divisor set of \(f\), write \(f \sim \prodtuple{\polyset}{m}\)
  (cf.~Remark~\ref{remark:polyset-egal}), and
  set $F = \frac{f}{p^n}$.  Let $\tilde F$ be a non-constant factor
  of $F^j$ in \(\Int(R)\) for some \(j \in \N\).  We show that
  $\tilde F$ is itself (associated to) a power of $F$.
  Note that for $j=1$ this implies that $F$ is irreducible. 

  It follows from Fact~\ref{fact:iv-fac}, with the abbreviated
  notation for polynomial products (Notation~\ref{notation:set-powers}),
  that \(\tilde{F} \sim \frac{\prodtuple{\polyset}{k}}{p^\ell}\) for some
  \(\ell \in \N_0\) and $\tuple 0\neq \tuple{k} \in \N_0^\polyset$.

  We can apply Lemma~\ref{lemma:factors-and-kernelA} to infer
  \begin{equation*}
    \tuple{k} - \frac{\ell}{n}\tuple{m} \in \fdk(f) = \boldsymbol 0
    \quad\implies\quad \frac{\ell}{n}\tuple{m} = \tuple{k}\in\N_0^\polyset.
  \end{equation*}
  Since $f$ is assumed not a proper power of another polynomial in
  $R[x]$, we must have\footnote{to be understood as the \(\gcd\) of
    all components of \(\tuple{m}\)} $\gcd(\tuple{m}) = 1$ and hence
  $\ell = t n$ for some $t \in \N_0$ and
  $\tuple{k} = t \tuple{m}$.  It follows that
  $\tilde F \sim \Big(\frac{\prodtuple{\polyset}{m}}{p^n}\Big)^t
  \sim F^t$, which was
  to be shown.

  Conversely, if $\fdk(f) \neq \boldsymbol 0$, then $\frac{f}{p^n}$ is not
  absolutely irreducible according to Theorem~\ref{theorem:nuf-bound}.
\end{proof}

\begin{remark}
  Let $D$ be a Dedekind domain, $P$ a maximal ideal of $D$ with finite
  index, and $F = \frac{f}{c} \in \Int(D)$ with $f\in D[x]$ primitive,
  non-constant, not a proper power of another polynomial and
  $0\neq c\in D$. Since $\Int(D) \subseteq \Int(D_P)$
  (see~\cite[Theorem~I.2.3]{Cahen-Chabert:1997:book}), we can regard
  $F$ as integer-valued polynomial over the discrete valuation domain
  $D_P$. If $\fdk(f) = \boldsymbol 0$, then $F$ is absolutely irreducible in
  $\Int(D_P)$ and a straight-forward argument (see, for
  example,~\cite[Corollary~6.11]{Frisch-Nakato-Rissner:2022:split})
  shows that $F$ is absolutely irreducible in $\Int(D)$,
\end{remark}

\begin{remark}
  We point out that there is a loose relation between the fdp matrices
  introduced here and the partition matrices defined in the paper of
  Frisch, Nakato and
  Rissner~\cite[Definition~5.5]{Frisch-Nakato-Rissner:2022:split} on
  absolute irreducibility of completely split polynomials in
  $\Int(R)$. Partition matrices, however, have trivial kernel
  (\cite[Proposition~6.5]{Frisch-Nakato-Rissner:2022:split}). If the
  split integer-valued polynomial to which a partition matrix
  \(A\) is associated is absolutely irreducible, \(A\) is indeed
  an fdp matrix upon suitable interpretation of row and column sets.
  This is not the case for split polynomials which are not absolutely
  irreducible.
  \end{remark}


%% file: bounds.tex
As pointed out in the introduction, for $F\in \Int(R)$ of the form
$\frac{f}{p}$ it suffices to check whether $F^2$ factors uniquely to
verify absolute irreducibility.  This motivates the question whether
an analogous statement holds if the denominator of $F$ is not
square-free: Is it possible to conclude that $F$ is absolutely
irreducible whenever $F^k$ factors uniquely for $1\le k \le S$ for
some $S$ which is yet to be determined?

Theorem~\ref{theorem:nuf-bound} provides an upper bound for such $S$ for an
integer-valued polynomial $F$ depending on the (absolute) values of
coordinates of integer elements in the fixed divisor kernel of
the numerator polynomial \(f\) of $F$.

In the following, we will provide, among others, a bound $S$ 
which only depends on the valuation of the denominator of $F$
and the size of the finite residue field of the underlying
discrete valuation domain $R$. To do so, we look for an upper bound for
\begin{equation*}
  \min\!\left\{\norm{\tuple{v}}_{\infty} \longmid \tuple 0\neq \tuple{v}\in \fdk(f)\cap \Z^\polyset \right\},
\end{equation*}
which  motivates the following definition.
\begin{definition}\label{definition:reduced-fdp}
  Let $f$ be as in notation of Convention~\ref{convention:R}. We call an fdp
  matrix of \(f\) \emph{reduced} if its rows are $\Q$-linearly
  independent.
\end{definition}
We hence look for an upper bound for
\begin{equation*}
  \min\!\left\{\norm{\tuple{v}}_{\infty} \longmid \tuple 0\neq \tuple{v}\in \ker(A)\cap \Z^\polyset \right\}
\end{equation*}
where $A$ is a reduced fdp matrix of $f$. To this end, we use a
Siegel-type lemma which we tailored to the type of matrices which
occur as (reduced) fdp matrices. The proof below is essentially the
same as the one in Baker's textbook~\cite[Lemma~1,
Chapter~3]{Baker:1990:transnth}, the modifications merely take the
additional assumptions on the system matrix into account.

\begin{lemma}[{Adapted version of Siegel's
    lemma}]\label{lemma:siegel-version}
  Let $r\in\N_0$, $s\in\N$, and $I$, $J$ be sets with $\card{I}=r$, $\card{J}=r+s$.
  Let $n \in \N$ and $A = (a_{i,j})_{(i,j)\in I \times J} \in \N_0^{I \times J}$ be such that
  $\sum_{j\in J} a_{i,j} \le n$ for all $i\in I$.
  Assume further that $A$ contains $u$ rows with exactly one non-zero entry.

  Then
  \begin{equation*}
    \min\!\left\{\norm{\tuple{v}}_{\infty} \longmid \tuple 0\neq \tuple{v}\in \ker(A)\cap \Z^J \right\} \le
    \left\lfloor n^{\frac{r-u}{s}}\right\rfloor.
  \end{equation*}
\end{lemma}
\begin{proof}
 Assume first $u = 0$, set $b = \lfloor n^{\frac{r}{s}}\rfloor$ and
 $X = \left\{\tuple{x}\in\N_0^J \longmid \norm{\tuple{x}}_{\infty}\le b\right\}$.
 Then $\tuple{x}\in X$ implies
 \begin{equation}
   0 \le (A\tuple{x})_i = \sum_{j\in J} a_{i,j}x_j \le nb
 \end{equation}
 for all $i\in I$.
 Hence, the set $\{A\tuple{x} \mid \tuple{x}\in X\}$ contains at most
 $(nb + 1)^r$ elements, whereas the cardinality of $X$ is equal to $(b+1)^{r+s}$.

 Since
 \begin{equation}\label{eq:pigeon}
  (b+1)^{r+s} = (b+1)^{s}(b+1)^r > n^r(b+1)^r \ge (nb + 1)^r
 \end{equation}
 holds, it follows from the pigeonhole principle that there exist $\tuple{x}$,
 $\tuple{y}\in X$ with $\tuple{x}\neq \tuple{y}$ such that
 $A\tuple{x} = A\tuple{y}$. Hence $\tuple{v} = \tuple{x}-\tuple{y} \neq \tuple 0$
 is an element in $\ker(A)$ whose every coordinate has absolute value
 at most $b$.

 Now turn to the general case $u\ge 0$. Let $E$ be the set of row indices
 for which the corresponding rows contain exactly one non-zero entry and set
 $C=\{j \in J \mid \exists e \in E \colon a_{e,j}\neq 0\}$.
 The definition of $E$ implies $\card{C} \le \card{E}\le u$.  Consider
 now the matrix $B$ which is obtained from $A$ by deleting the $u$
 rows corresponding to $E$ as well as the columns corresponding to $C$.
 Then $B$ has $r-u$ rows and $r+s-\card{C}$ columns. 
 Note that if $\tuple{v} \in \ker(A)$ and $k\in C$, then $v_k = 0$. Hence
 \begin{align*}
   \min\!\left\{\norm{\tuple{v}}_{\infty} \longmid \tuple 0\neq \tuple{v}\in \ker(A)\cap \Z^J\right\}
   &= \min\!\left\{\norm{\tuple{v}}_{\infty} \longmid \tuple 0\neq \tuple{v}\in \ker(B)\cap \Z^{J\setminus C}\right\} \\
   &\le \left\lfloor n^{\frac{r-u}{s+u-\card{C}}}\right\rfloor
     \le \left\lfloor n^{\frac{r-u}{s}}\right\rfloor
 \end{align*}
 by applying the result of the first part to $B$.
\end{proof}

\begin{reptheorem}{theorem:upperbounds}
  \input{theorem-3.tex}
\end{reptheorem}

\begin{remark}
  It follows from Theorem~\ref{theorem:upperbounds} that $F$ is
  \textbf{not} absolutely irreducible if and only if $F^j$ factors
  non-uniquely for some
  $1\le j\le 2(n+1)n^{q^{\left\lceil \frac{n}{2} \right\rceil}}$. Note
  that this upper bound only depends on the fixed divisor of $f$ and
  the size of the residue class field, but not the reduced fdp
  matrix $A$.
\end{remark}

\begin{proof}
  Let \(r = \rank(A) = \card{W}\).  By definition,
  \ref{bound:1} implies~\ref{bound:2} and~\ref{bound:3}.
  Moreover, by Lemma~\ref{lemma:matrix-rows}, $A$ has at most
  $q^{\left\lceil \frac{n}{2} \right\rceil}$ rows and hence
  $r-u \le r \le q^{\left\lceil \frac{n}{2} \right\rceil}$, so
  that~\ref{bound:2} implies~\ref{bound:3}.

  We prove by contraposition that~\ref{bound:1} follows
  from~\ref{bound:3}. Assume that $F$ is irreducible, but not
  absolutely irreducible.  Then $\fdk(f) = \ker(A) \neq \boldsymbol 0$ by
  Theorem~\ref{theorem:abs-irred-and-kernel} or, equivalently,
  $r < \card \polyset$. We apply Lemma~\ref{lemma:siegel-version} to
  $A$: There exists $\tuple 0\neq \tuple{v}\in \ker(A) \cap \Z^{\polyset}$
  with
  $\norm{\tuple{v}}_{\infty} \le \left\lfloor n^{\frac{r-u}{\card
        \polyset-r}} \right\rfloor \le n^{r-u}$. Since
  $\tuple{v}\in \fdk(f) \cap \Z^{\polyset}$ and $F$ is assumed to be
  irreducible, it follows from Theorem~\ref{theorem:nuf-bound} that
  $F^j$ factors non-uniquely for all
  $j\ge (n+1) \left(\Zinfnorm{\frac{\tuple{v}^+}{\tuple{m}}}
  + \Zinfnorm{\frac{\tuple{v}^-}{\tuple{m}}}\right)$.  Because
  of
  \begin{equation*}
  S = 2(n+1)n^{r-u} \ge 2(n+1)\norm{\tuple v}_{\infty} \ge
  (n+1)\left(\Zinfnormb{\frac{\tuple{v}^+}{\tuple{m}}} +
  \Zinfnormb{\frac{\tuple{v}^-}{\tuple{m}}}\right),
  \end{equation*}
  it follows that $F^S$ factors non-uniquely so that~\ref{bound:3}
  does not hold.
\end{proof}


%% file: bound-tightness.tex
By Theorem~\ref{theorem:nuf-bound}, we know that $F = \frac{f}{p^n}$
is not absolutely irreducible and, more accurately, that $F^j$ factors
non-uniquely whenever $\tuple 0\neq\tuple{v}\in \fdk(f)$ and \(j \in \N\) with
$j \ge (n+1) \left(\Zinfnorm{\frac{\tuple{v}^+}{\tuple{m}}} +
\Zinfnorm{\frac{\tuple{v}^-}{\tuple{m}}}\right)$.

We show below, in Theorem~\ref{theorem:realization}, that this bound
cannot be improved in general.
Indeed, we show that for all integers $n\ge 2$ there exists a
polynomial $f\in R[x]$ with $\val(\fixdiv(f)) = n$ such that
\(\tuple{m} = \tuple{1} = (1)_{g\in \polyset}\), $F = \frac{f}{p^n}$
is irreducible and factors uniquely up to the bound \((n+1)K\) with
\begin{equation*}
  K = \min\!\left\{\norm{\tuple{v}^+}_\infty + \norm{\tuple{v}^-}_\infty
    \longmid \tuple 0\neq \tuple{v} \in \ker(A)\cap \Z^\polyset\right\},
\end{equation*}
where $A$ is a (reduced) fdp matrix of $f$. As already mentioned in
the introduction, the case $n = 1$ has been covered where the bound is
known to be $2$.

By Lemma~\ref{lemma:siegel-version} (variation of Siegel's lemma)
above, $K\le 2n^{r-u}$ whenever \(A\) contains $r$ rows, $u$ of which
contain exactly one non-zero entry. In Theorem~\ref{theorem:realization}, we furnish
examples for every $r\ge 2$ such that
\begin{enumerate}
	\item $u = 1$ and
	\item $K = (n-1)^{r-1} + (n-1)^{r-2}$.
\end{enumerate}
It remains open whether this is actually optimal, in the sense
that
\begin{equation*}
	(n-1)^{r-1} + (n-1)^{r-2} = \max_{A \in \mathcal{A}}
	\min \!\left\{\norm{\tuple{v}^+}_{\infty} + \norm{\tuple{v}^-}_{\infty}
	\longmid \tuple 0\neq \tuple{v} \in \ker(A)\cap \Z^\polyset\right\},
\end{equation*}
where \(\mathcal{A}\) denotes the set of all reduced fdp matrices
\(A\) of \(f\) with exactly one row with only one non-zero entry.

The next proposition establishes sufficient conditions on the
polynomial $f$ in order to achieve the asserted bound, whereas in
Theorem~\ref{theorem:realization} we show that such polynomials can
always be realized.

\begin{proposition}\label{theorem:realization-prep}
  Let $r$, $n\ge 2$ be integers. Using the notation of
  Convention~\ref{convention:R}, let $f=\prod_{g\in \polyset}g$ where
  $\polyset = \{g_1, \dots, g_{r+1}\}$ is an irreducible divisor
  set, $\val(\fixdiv(f)) = n$, and
  \begin{equation*}
    A =
    \begin{pmatrix}
      1 & n-1 &        &       &      &     \\
        & 1   & n-1    &       &      &     \\
        &     &  \ddots& \ddots&      &     \\
        &     &        &   1   & n-1  &  \\
        &     &        &       &      & n   \\
    \end{pmatrix} \in \N_0^{W \times \polyset}
  \end{equation*}
  is a reduced fdp matrix of \(f\) for some $W\subseteq \witnesses(f)$
  with $\card W = r$, where the \(i\)-th column corresponds to the
  polynomial \(g_i\).  In addition, we assume that there exist
  \(a_1\), \(a_2 \in R\) such that $\val(g_i(a_i)) = n+1$ and
  $\val(g_j(a_i)) = 0$ for all \(i = 1\), $2$ and \(1\le j \le r+1\)
  with \(j \neq i\).

  Then \(F\) is irreducible, but \textbf{not} absolutely irreducible
  in \(\Int(R)\). Indeed, the minimal exponent $S$ such that $F^S$
  does not factor uniquely satisfies the two equations:
  \begin{enumerate}
  \item
    $S = (n+1)\left((n-1)^{r-1} + (n-1)^{r-2}\right)$ with
    $r =\rank(A)$ and
  \item
    $S= (n+1)\min\!\left\{\norm{\tuple{v}^+}_\infty +
      \norm{\tuple{v}^-}_\infty \longmid \tuple 0\neq \tuple{v} \in
      \fdk(f)\cap \Z^\polyset\right\}$ (which is the minimal lower
    bound given in Theorem~\ref{theorem:nuf-bound} applied to the
    current setting).
  \end{enumerate}
\end{proposition}
\begin{proof}
  For readability, we write \(i \in [r+1]\) whenever we address
  \(g_i \in \polyset\).

  First, we determine $\ker(A)$.  Let
  $\tuple{u} = (u_i)_{i=1}^{r+1}\in \ker(A)$. It is immediately seen
  that $u_{r+1} = 0$.  Moreover, a straight-forward computation yields
  that $u_i = -(n-1)u_{i+1}$ for all $1\le i\le r-1$. Therefore,
  $u_i = (-1)^{r-i}(n-1)^{r-i}u_r$. Thus,
  $\ker(A) = \Qspan(\tuple{v})$ where $\tuple{v} = (v_i)_{i=1}^{r+1}$
  with
  \begin{equation}\label{eq:vcomp}
    v_i =
    \begin{cases}
      (-1)^{r-i}(n-1)^{r-i}, & 1 \le i \le r, \\[1ex]
      0,                    & i = r+1.
    \end{cases}
  \end{equation}
  This implies $\dim(\ker(A)) = 1$ and hence $\rank(A) = r$, as
  claimed.

  Next, we set
  \begin{equation*}
    K = \min\!\left \{\norm{\tuple{v}^+}_\infty +
    \norm{\tuple{v}^-}_\infty \mid \tuple 0\neq \tuple{v} \in \ker(A)\cap
    \Z^\polyset\right\}.
  \end{equation*}

  Note that Theorem~\ref{theorem:nuf-bound} guarantees that
  $F^{(n+1)K}$ factors non-uniquely. In order to prove the whole
  assertion of the theorem, we need to show that
  \begin{enumerate}[label=(\alph*)]
  \item\label{pi:4-1} $S = (n+1)K$ is the smallest power of $F$
    factoring non-uniquely,
  \item\label{pi:4-2} $F$ is irreducible, and
  \item\label{pi:4-3} $K = (n-1)^{r-1} + (n-1)^{r-2}$.
  \end{enumerate}

  We infer $\ker(A)\cap\Z^{r+1} = \Z \tuple{v}$ by considering the
  $r$-th component of $\lambda \tuple{v}$ for $\lambda\in\Q$. From
  Equation~\eqref{eq:vcomp}, we conclude that $v_1$ and $v_2$ have
  opposite signs and $\abs{v_1} \ge \abs{v_2} \ge \abs{v_i}$ for
  $i > 2$.  (Equality is only possible for $n = 2$.) Hence
  \begin{equation*}
    \norm{\lambda \tuple{v}^+}_{\infty} + \norm{\lambda \tuple{v}^-}_{\infty} =
    \abs{\lambda}\left(\abs{v_1} + \abs{v_2}\right)\ge (n-1)^{r-1} + (n-1)^{r-2}
  \end{equation*}
  for all $\lambda\in\Z$, $\lambda\neq 0$. Since
  $\tuple 0 \neq \tuple{v}\in\ker(A)$ satisfies this with equality, we have
  proven Item~\ref{pi:4-3} from the list above.

  Let $j\in \N$ and $F_1$, $F_2 \in \Int(D)$ (not necessarily
  irreducible) such that $F^j = F_1F_2$.  By Fact~\ref{fact:iv-fac},
  we know that
  \begin{equation}\label{eq:nontrivialfactorization}
    F_1 \sim \frac{\prodtuple{\polyset}{k}}{p^\ell} = \frac{\prod_{i=1}^{r+1}g_i^{k_i}}{p^\ell}
    \quad\text{and}\quad
    F_2\sim \frac{\prodnotuple{\polyset}{j\tuple{1} - \tuple{k}}}{p^{jn-\ell}}  = \frac{\prod_{i=1}^{r+1}g_i^{j-k_i}}{p^{jn-\ell}}
  \end{equation}
  for some $0\le \ell \le jn$,
  \(\tuple{k} = (k_i)_{i=1}^{r+1} \in \N_0^{r+1}\) with
  \(\tuple{k} \le j\tuple{1}\) such that
  \begin{equation*}
    \val(\fixdiv(\prodtuple{\polyset}{k})) \ge \ell \quad\text{and}\quad
    \val(\fixdiv(\prodnotuple{\polyset}{j\tuple{1} - \tuple{k}})) \ge jn-\ell.
  \end{equation*}
  We apply Lemma~\ref{lemma:factors-and-kernelA} for the factor $F_1$
  of $F^j$ and conclude that
  \begin{equation}\label{eq:l>0}
    \tuple{k} =  \frac{\ell}{n} \tuple{1}  + \lambda \tuple{v}
  \end{equation}
  for some $\lambda \in \Q$.

  Observe that $k_i\in \N_0$ and hence if $\lambda = 0$, then $F_1$ is
  necessarily a power of $F$ (possibly the $0$-th if $\ell = 0$),
  resulting in a trivial factorization of $F^j$. Moreover, $\ell=0$
  immediately implies $\lambda = 0$, since $k_i\ge 0$ and $\tuple{v}$
  has positive and negative components.

  Thus, we can safely assume $\ell > 0$ and $\lambda\neq 0$ from now
  on.  From \eqref{eq:l>0} and $v_{r+1}=0$, we infer
  \begin{equation*}
    0 < \frac{\ell}{n} = \frac{\ell}{n} + \lambda v_{r+1} = k_{r+1} \in \N
  \end{equation*}
  and therefore that $n$ divides $\ell$. We set $k = k_{r+1}$, so that
  $\ell = kn$.

  This allows us to address Item~\ref{pi:4-2} and show that $F$ is
  irreducible. We consider non-trivial factorizations of $F$ itself,
  that is, $j=1$. Then $\tuple{k} \le \tuple{1}$, which implies
  $k = k_{r+1} = 1$ and hence $\ell = kn = n$. Given the structure of
  $A$, it follows that for each choice of $J\subsetneq [r+1]$ there
  exists a fixed divisor witness $w$ such that
  \begin{equation*}
    \val\!\left(\prod_{i\in J}g_i(w)\right) < n,
  \end{equation*}
  which, in combination with
  $\val(\fixdiv(\prodtuple{\polyset}{k})) \ge \ell = n$, implies that
  each $g_i$ has to appear as a factor of the numerator of $F_1$. In
  other words, $\tuple{k} = \tuple{1}$ and hence $F_1 = F$, that is,
  $F$ is irreducible.

  It remains to prove Item~\ref{pi:4-1}: Whenever $F^j$ factors
  non-uniquely, then $j\ge (n+1)K$. At this point, we can assume that
  $j\ge 2$. Moreover, note that $\tuple{k} \neq e\tuple{1}$ for all
  $e\in \N_0$, since $F_1F_2$ is a non-trivial factorization of $F$. To
  simplify the following arguments, we also assume $\lambda v_1 > 0$
  (by transition from $\tuple{v}$ to $-\tuple{v}\in \ker(A)$ if
  necessary).  Since $v_1$ and $v_2$ have different signs, we obtain
  $\lambda v_2 < 0$.

  Further, Equation~\eqref{eq:l>0} yields
  \begin{equation*}
    k_r = k +\lambda v_r = k + \lambda,
  \end{equation*}
  which implies $\lambda \in \Z$, as $k_r$ and $k$ are integers.

  Also, we conclude from Equation~\eqref{eq:l>0}
  \begin{equation*}
    k_{1} = k + \lambda v_1  \quad\text{and}\quad k_2 = k + \lambda v_2.
  \end{equation*}
  By hypothesis, there exist elements $a_1$, $a_2 \in R$ such that
  $\val(g_i(a_i)) = n+1$ and $\val(g_j(a_i)) = 0$ for $1\le i\le 2$
  and all \(1 \le j \le r+1\) with \(j \neq i\).  This further implies
  \begin{equation*}
    \val\!\left(\prodtuple{\polyset}{k}(a_2)\right) = k_2\val(g_2(a_2)) = (k + \lambda v_2)(n+1) \ge \ell.
  \end{equation*}
  Since $\ell = kn$ and $\lambda v_2 < 0$ by assumption, we conclude
  that
  \begin{equation}\label{eq:estimate-k}
    k \ge -\lambda v_2 (n+1) = \abs{\lambda v_2} (n+1).
  \end{equation}
  Similarly, we evaluate $f_2$ at $a_1$ to see
  \begin{equation*}
    \val\!\left(\prodnotuple{\polyset}{j\tuple{1} - \tuple{k}}(a_1)\right) = (j-k_1)\val(g_1(a_1)) = (j - k - \lambda v_1)(n+1) \ge jn - \ell.
  \end{equation*}
  This further implies
  \begin{equation}\label{eq:estimate-j-k}
    j-k \ge \lambda v_1 (n+1) = \abs{\lambda v_1} (n+1).
  \end{equation}
  Summing up Equations~\eqref{eq:estimate-k}
  and~\eqref{eq:estimate-j-k} yields
  \begin{equation*}
    j \ge (n+1)\abs{\lambda} (\abs{v_1} + \abs{v_2}) = \abs{\lambda} (n+1)K.
  \end{equation*}
  Due to $\lambda \in \Z$, \(\lambda \neq 0\), we obtain $j\ge (n+1)K$
  whenever $F^j$ factors non-uniquely, which completes the proof.
\end{proof}

For the remaining part of this section, we show that for every choice
of integers $r$, $n\ge 2$, there exists a discrete valuation domain
$R$ with finite residue field and a set of irreducible, non-constant
polynomials $\polyset$ which satisfy the hypotheses of
Proposition~\ref{theorem:realization-prep}. For the construction, we
need a lemma that allows us to simultaneously replace a family of
polynomials by ``irreducible variants'' which exhibit a similar
behavior concerning the valuations when evaluating at elements of
$R$. The following result is a slight variation of
\cite[Lemma~3.3]{Frisch-Nakato-Rissner:2019:Sets-of-lengths}. Note
that the proof is almost identical and only differs in the fact that
we want to control valuations up to some prespecified point and not
only up to the fixed divisor. The original result
\cite[Lemma~3.3]{Frisch-Nakato-Rissner:2019:Sets-of-lengths} follows
as a special case, as described in Remark~\ref{remark:orig-assertion}
below.

\begin{lemma}[{Variation of \cite[Lemma~3.3]{Frisch-Nakato-Rissner:2019:Sets-of-lengths}}]\label{lemma:irreducible-replacements}
  Let $D$ be a Dedekind domain with infinitely many maximal ideals and
  $K$ its quotient field. Further, let $I\neq \emptyset$ be a finite
  set and $h_i\in D[x]$ for $i\in I$ be monic, non-constant
  polynomials and set $d = \sum_{i\in I} \deg(h_i)$.

  Then, for every $n\in \N_0$, there exist monic polynomials
  $g_i\in D[x]$ for $i\in I$ such that
  \begin{enumerate}
  \item\label{rep:1} $\deg(g_i) = \deg(h_i)$ for all $i\in I$,
  \item\label{rep:2} the polynomials $g_i$ are irreducible in $K[x]$
    and pairwise non-associated in $K[x]$, and
  \item\label{rep:3} $g_i \equiv h_i \mod P^{n+1}D[x]$ for every
    maximal ideal $P$ of $D$ of index at most~$d$.
  \end{enumerate}
\end{lemma}

\begin{remark}\label{remark:pull-to-R}
  Note that for any Dedekind domain $D$ and prime ideal $P$ of $D$: If
  $g \equiv h \mod P^{n+1}D[x]$, then
  $g(a) \equiv h(a) \mod P^{n+1}$ for all $a\in D$.
\end{remark}

\begin{remark}\label{remark:orig-assertion}
  The original
  assertion~\cite[Lemma~3.3.]{Frisch-Nakato-Rissner:2019:Sets-of-lengths}
  follows immediately by choosing
  $n = \max\!\left\{\val_P\!\left(\prod_{i\in I}h_i\right) \longmid P
    \text{ prime ideal of } D \text{ with } \card{D/P} \le d\right\}$
  and observing that
  Lemma~\ref{lemma:irreducible-replacements}\ref{rep:3} together with
  Remarks~\ref{remark:equal-rows} and~\ref{remark:pull-to-R} imply
  that
  \begin{equation*}
    \fixdiv\!\left(\prod_{i\in J_1} g_i\prod_{j\in J_2}h_i\right) = \fixdiv\!\left(\prod_{i\in I}h_i\right)
  \end{equation*}
  holds for all partitions $J_1\uplus J_2 = I$.
\end{remark}
\begin{proof}
  Let $P_1$, \ldots, $P_m$ be all maximal ideals of $D$ whose
  respective residue fields have cardinality less than or equal to
  $d$, cf.~\cite[Proposition~13]{Samuel:1971:eucl-rings} as to why
  there are only finitely many.

  Then there exist $e_1$, \ldots, $e_m\in \N_0$ such that
  \begin{equation*}
    \fixdiv\!\left( \prod_{i\in I}h_i \right) = \prod_{i=1}^mP_i^{e_i}.
  \end{equation*}

  Let $Q$ be a prime ideal of $D$ different from any of the $P_i$
  above. Further, for $i\in I$, let
  $h_i = x^{d_i} + \sum_{j=0}^{d_i-1} h_{i,j}x^j$ with $d_i\in \N$ and
  $h_{i,j}\in D$ the coefficient of $x^j$ of $h_i$.  By the Chinese
  Remainder Theorem, there exist $c_{i,j}\in D$ for $i\in I$ and
  $0\le j \le d_i-1$ such that
  \begin{enumerate}[label=(\alph*)]
  \item $c_{i,j} \in \prod_{i=1}^m{P_i^{e_i + n + 1}}$ for all
    $i\in I$ and $0\le j \le d_i-1$,
  \item $c_{i,j} \equiv -h_{i,j} \mod Q$ for all $i\in I$ and
    $0\le j \le d_i-1$, and
  \item $c_{i,0} \not\equiv -h_{i,0} \mod Q^2$ for all $i\in I$.
  \end{enumerate}
  These conditions determine the elements $c_{i,j}$ only modulo
  $Q^2\prod_{i=1}^m{P_i^{e_i + n + 1}}$. This allows us to choose the
  elements in a way such that
  $c_{i,0} + h_{i,0} \neq c_{j,0} + h_{j,0}$ for $i\neq j$.

  We set
  \begin{equation*}
    g_i = h_i + \sum_{j=0}^{d_i-1} c_{i,j}x^j.
  \end{equation*}
  By construction, the polynomials $g_i$ are monic and satisfy
  Assertion~\ref{rep:1}.  Moreover, by choice of $c_{i,j}$, it follows
  that
  \begin{equation*}
    g_i \equiv h_i \mod\prod_{i=1}^m{P_i^{e_i + n + 1}}D[x],
  \end{equation*}
  which, in turn, implies Assertion~\ref{rep:3}.

  Finally, by construction, the polynomials $g_i$, $i\in I$, are
  irreducible in $D[x]$ according to Eisenstein's criterion
  (cf.~\cite[§29, Lemma~1]{Matsumura:1989:comm-alg}). Since the $g_i$
  are monic and $D$ is integrally closed, it follows that the
  polynomials $g_i$ are also irreducible in $K[x]$ (cf.~\cite[Ch.~5,
  §1.3, Prop.~11]{Bourbaki:1989:comm-alg}). The choice of $c_{i,0}$
  guarantees that the $g_i$ are pairwise non-associated in
  $K[x]$. Hence Assertion~\ref{rep:2} holds, which completes the
  proof.
\end{proof}

\begin{reptheorem}{theorem:realization}
  \input{theorem-4.tex}
\end{reptheorem}
\begin{remark}
  As to the dependence of $R$ and $F$ on $r$, we point out that the
  size of the residue field of $R$ in the construction below is at
  least $r+2$ and the polynomial $f$ has $r+1$ non-associated,
  irreducible factors.
\end{remark}
\begin{proof}
  Let $p$ be a prime number with $p\ge r+2$ and set $R = \Z_{(p)}$. In
  order to prove the assertions, we construct polynomials $g_1$,
  \ldots, $g_{r+1}$ to satisfy the hypotheses of
  Proposition~\ref{theorem:realization-prep}. The construction takes
  place over the ring $\Z$ instead of $\Z_{(p)}$, as we would like to
  invoke Lemma~\ref{lemma:irreducible-replacements} to find
  $g_j\in \Z[x]$ with the desired properties.

  We choose $w_1$, $w_2$, \ldots, $w_r$, $a_{1}$, $a_{2}$, $z_{r+3}$,
  \ldots, $z_p\in \Z$ to be a complete system of residues modulo
  $p \ge r+2$. For simplicity, we assume that this choice does not
  contain a complete set of residues modulo any prime less than $p$.

  Then there exist $b_1$, \ldots, $b_r$, $c_1$, $c_2\in \Z$ satisfying
  the conditions
  \begin{enumerate}
  \item $\val(b_i - w_i) = 1$ for all $1\le i \le r$ and
  \item $\val(c_i - a_i) = 1$ for $i=1$, 2.
  \end{enumerate}
  We set
  \begin{equation*}
    h_j =
    \begin{cases}
      (x-b_1)(x-c_1)^{n+1}                  & j = 1,          \\
      (x-b_1)^{n-1}(x-b_2)(x-c_2)^{n+1}      & j = 2,          \\
      (x-b_{j-1})^{n-1}(x-b_j)               & 3 \le j \le r-1,\\
      (x-b_{r-1})^{n-1}                      & j=r, \text{ and} \\
      (x-b_r)^{n}\prod_{i=r+3}^p(x-z_i)^{n+1} & j = r+1.
    \end{cases}
  \end{equation*}
  By construction, it follows that
  \begin{equation}\label{eq:argue-equality}
    \val(h_j(w_i)) =
    \begin{cases}
      1   & i=j \text{ and } 1 \le j \le r-1, \\
      n-1 & i=j-1 \text{ and } 2 \le j \le r, \\
      n   & i=r \text{ and } j = r+1, \\
      0   & \text{otherwise},
    \end{cases}
  \end{equation}
  and, for $\ell=1$, $2$,
  \begin{equation*}
    \val(h_j(a_\ell)) =
    \begin{cases}
      n+1   & j = \ell, \text{ and} \\
      0     & j\neq \ell  \text{ and } 1 \le j \le r+1. \\
    \end{cases}
  \end{equation*}
  Moreover, if $a\in R$, then $a$ is congruent to exactly one of the
  $w_i$, $a_\ell$, or $z_k$ with $1 \le i \le r$, $\ell=1$, $2$, and
  $r+3 \le k \le p$ modulo $pR$. This implies that $a$ is congruent to
  one of the elements $b_i$, $c_\ell$, or $z_k$.  Thus, for all
  \(a \in R\),
  \begin{equation*}
    \val\!\left(\prod_{j=1}^{r+1} h_j(a)\right)
    \ge
    \begin{cases}
      1 + (n-1)  & a \equiv b_i \mod pR \text{ with } 1\le i \le r-1,\\
      n          & a \equiv b_r \mod pR,\\
      n + 1      & a \equiv c_\ell \text{ or } a \equiv z_k \text{ with  \(\ell=1\), 2, \(r+3 \le k \le p\)},
    \end{cases}
  \end{equation*}
  and hence, in combination with Equation~\eqref{eq:argue-equality},
  \begin{equation*}
    \val\!\left(\fixdiv\!\left(\prod_{j=1}^{r+1}h_j\right)\right) = n.
  \end{equation*}

  We apply Lemma~\ref{lemma:irreducible-replacements} to find monic
  polynomials $g_j\in \Z[x]$ for $1\le j \le r+1$ which are
  irreducible in $\Q[x]$ such that for all $1\le j \le r+1$, we have
  \begin{equation}\label{eq:g-h-valuations-Z}
    g_j \equiv h_j \mod p^{n+2}\Z[x].
  \end{equation}
  Note that all $g_j$ are irreducible in $\Q[x]$ by
  Lemma~\ref{lemma:irreducible-replacements}, and hence all $g_i$ are
  irreducible in $\Z_{(p)}[x] \subseteq \Qx$.  Moreover,
  from~\eqref{eq:g-h-valuations-Z} we conclude that
  \begin{equation*}
    g_j \equiv h_j \mod p^{n+2}\Z_{(p)}[x]
  \end{equation*}
  holds, which, in turn, immediately implies
  \begin{equation}\label{eq:g-h-valuations}
    g_j(a) \equiv h_j(a) \mod p^{n+2}\Z_{(p)}
  \end{equation}
  for all $a\in \Z_{(p)} = R$ and all $1\le j\le r+1$.

  We set $f = \prod_{j=1}^{r+1}g_j$. By
  Remark~\ref{remark:equal-rows}, in combination with
  Congruence~\eqref{eq:g-h-valuations}, it follows that
  $\val(g_j(w_i)) = \val(h_j(w_i))$ and
  $\val(g_j(a_\ell)) = \val(h_j(a_\ell))$ for all $1\le j \le r+1$,
  $1\le i \le r$, and $\ell=1$, $2$. This implies that $g_1$, \ldots,
  $g_{r+1}$ satisfy the hypotheses of
  Proposition~\ref{theorem:realization-prep}. The assertion follows.
\end{proof}
